\documentclass[12pt]{article}

\usepackage[T1]{fontenc}
\usepackage[margin = 1in]{geometry}
\usepackage{setspace}

\usepackage{amsthm, amssymb, amstext}
\usepackage[fleqn]{amsmath}
\usepackage{latexsym}
\usepackage[dvips]{graphicx}
\usepackage{comment}
\usepackage{hyperref}
\usepackage[capitalize]{cleveref}
\usepackage{mathtools}
\usepackage{enumerate} 
\usepackage{paralist}
\usepackage{enumitem}
\usepackage{bm}

\usepackage{todonotes}
\usepackage{comment}
\usepackage{tikz}
\usetikzlibrary{graphs, graphs.standard, positioning}
\usetikzlibrary{decorations.pathmorphing}

\crefname{claim}{Claim}{Claims}
\crefname{conjecture}{Conjecture}{Conjectures}
\crefname{question}{Question}{Questions}
\crefname{figure}{Figure}{Figures}

\newenvironment{proofofclaim}{\noindent \textsc{Proof of Claim:}}{\unskip\nobreak\hfill$\Diamond$\medskip}

\usepackage{setspace}
\makeatletter
\newtheorem*{rep@theorem}{\rep@title}
\newcommand{\newreptheorem}[2]{%
\newenvironment{rep#1}[1]{%
 \def\rep@title{#2 \ref{##1}}%
 \begin{rep@theorem}}%
 {\end{rep@theorem}}}
\makeatother

\newreptheorem{theorem}{Theorem}
\newreptheorem{lemma}{Lemma}
\newreptheorem{corollary}{Corollary}

\newtheorem{theorem}{Theorem}
\newtheorem{lemma}[theorem]{Lemma}

\newtheorem{question}{Question}

\newtheorem{claim}{Claim}

\theoremstyle{definition}

\newcommand{\TT}{\mathcal{T}}
\newcommand{\torso}[2]{#1\langle #2\rangle}

\newcommand\kk{6}
\pgfmathtruncatemacro\kkm{\kk-1}
\pgfmathtruncatemacro\kkmm{\kk-2}
\pgfmathtruncatemacro\kkmmm{\kk-3}

\newcommand\ct{11}

\pgfmathtruncatemacro\cth{(\ct-1)/2}
\pgfmathtruncatemacro\ctt{\cth * (\cth+1)/2}
\pgfmathtruncatemacro\cts{\ctt * \ctt}

\renewcommand\leq{\leqslant}

\renewcommand\geq{\geqslant}

\definecolor{myRed}{rgb}{0.68, 0.05, 0.0}
\colorlet{myBlue}{blue!70!black}
\colorlet{myViolet}{myBlue!55!myRed}
\definecolor{darkraspberry}{rgb}{0.53, 0.15, 0.34}
\definecolor{olive}{rgb}{0.42, 0.56, 0.14}

\hypersetup{
	colorlinks=true,
        linkcolor=myBlue, 
        citecolor=myBlue,
	bookmarksopen=true,
	bookmarksnumbered,
	bookmarksopenlevel=2,
	bookmarksdepth=3
}
\usepackage{breakurl}
\usepackage{authblk}

\title{Induced Minors and Region Intersection Graphs}

\author[1]{\'Edouard Bonnet}
\author[1]{Robert Hickingbotham}

\affil[1]{CNRS, ENS de Lyon, Université Claude Bernard Lyon 1, LIP, UMR 5668, Lyon, France}

\date{}

\begin{document}

\maketitle

\begin{abstract}
  We show that for any positive integers $g$ and $t$, there is a~$K_{\kk}^{(1)}$-induced-minor-free graph of girth at least~$g$ that is not a~region intersection graph over the class of $K_t$-minor-free graphs.
  This answers in a~strong form the recently raised question of whether for every graph~$H$ there is a~graph $H'$ such that $H$-induced-minor-free graphs are region intersection graphs over $H'$-minor-free graphs.
\end{abstract}

\section{Introduction}

Inspired by the success of Robertson and Seymour's graph minor theory~\cite{GraphMinors}, a~recent line of work aims to extend this theory to the realm of induced-minor-free classes.\footnote{All the relevant notions are defined in~\cref{sec:prelim}.}
Currently, far less is understood on classes excluding an induced minor than on those excluding a~minor. 
While $H$-minor-free \mbox{$n$-vertex} graphs are known since the 90's to have treewidth $O_H(\sqrt n)$~\cite{Alon90}, foreshadowed a~decade earlier by the Lipton--Tarjan planar separator theorem~\cite{Lipton79}, only recently were $H$-induced-minor-free \mbox{$m$-edge} graphs shown to have treewidth $\widetilde{O}_H(\sqrt m)$~\cite{KorhonenL24}.

There are several open questions (for simplicity, we phrase all of them as conjectures) on induced-minor-free classes.
\begin{compactitem}
\item For any planar graph $H$, the independence number of any $H$-induced-minor-free graph can be computed in polynomial time (see \cite[Question 8.2]{Dallard-3}).\footnote{Merely obtaining a~quasipolynomial-time algorithm is also a~wide open question.}
\item For any planar graph $H$, every $H$-induced-minor-free graph admits a~balanced separator dominated by a~subset of size $O_H(1)$ (Gartland--Lokshtanov's conjecture~\cite{GartlandThesis}). 
\item For any planar graph $H$, every $H$-induced-minor-free graph has treewidth at~most linear in its maximum degree (see~\cite{BonnetHKM23}).
\item For any graph $H$, the independence number admits a~polynomial-time approximation scheme in $H$-induced-minor-free graphs.
\item For any planar graph $H$, weakly sparse $H$-induced-minor-free classes have bounded twin-width (a~special case is mentioned in~\cite{Bonamy24}). 
\item For any planar graph $H$, every $H$-induced-minor-free graph has treewidth at~most linear in its Hadwiger number~(see \cite{CDDFGHHWWYl2024hadwiger}).
\item For any graph $H$, every $H$-induced-minor-free graph is quasi-isometric to an $H$-minor-free graph~(a more general conjecture is found in~\cite{GP2023Metric}).
\end{compactitem}
\medskip
All these questions are open within classes of large girth, a condition which may make them more approachable.
One more question, posed independently by Lokshtanov~\cite{Lokshtanov-pc} and McCarty~\cite{McCarty-lectures}, is whether region intersection graphs could provide a bridge between the structure of minors and induced minors. A graph $G$ is a~\emph{region intersection graph} (\emph{RIG}) over a~graph $H$ if there exists a collection $\mathcal{R}=(R_v \subseteq H \colon v\in V(G))$ of connected subgraphs of $H$ such that $uv \in E(G)$ if and only if $V(R_u)\cap V(R_v)\neq \emptyset$. We call $H$ the \emph{host graph} of $G$. 

\begin{question}\label{conj:main}
  Is every graph class excluding a~fixed induced minor included in the region intersection graphs of a~class excluding a~fixed minor?
\end{question}

If true, one could then work with the host graph and benefit from its decomposition given by the Graph Minor Structure Theorem~\cite{robertson2003graph}.
Wiederrecht asked a~related question of whether one can \emph{determine} if a~given induced-minor-free class is a~region intersection graph over a~minor-free class~\cite{Wiederrecht23op}.

Region intersection graphs were introduced by Lee~\cite{Lee17} as a generalization of the well-studied class of string graphs (intersection graph of curves on the plane). Indeed, a graph is a string graph if and only if it is a region intersection graph over some planar graph. The class of string graphs does not exclude any graph as a minor, but excludes any 1-subdivision of a~non-planar graph as an induced minor~\cite{Sinden66}. More generally, Lee~\cite{Lee17} proved the following relationship between region intersection graphs and minors.

\begin{lemma}[\cite{Lee17}]\label{No1Sub}
	For every graph $G$, if a graph $H$ is not a minor of $G$ then any graph that contains $H^{(1)}$ as an induced minor is not a region intersection graph over $G$.
\end{lemma}

Thus RIGs over an $H$-minor-free class are examples of classes excluding an induced minor.
The theory on region intersection graphs, and mainly on string graphs, is more advanced than that of induced-minor-free graphs. For instance, RIGs over $K_t$-minor-free classes can be $O_t(1)$-vertex-colored (or $O_t(1)$-edge-colored) such that every monochromatic connected component has bounded weak diameter \cite{Lee17,Davies-pc}.
Such a~result is useful in various contexts, and it would resolve several conjectures for classes excluding a~fixed induced minor (see for instance~\cite{KorhonenL24,BonnetHKM23}).
One way to achieve that would be via a~positive answer to~\cref{conj:main}.

Unfortunately, we answer \cref{conj:main} negatively, and perhaps more surprisingly, even within classes of arbitrarily large girth.

\begin{theorem}\label{thm:main}
  For any positive integers $t$ and $g$, there is a~$K_{\kk}^{(1)}$-induced-minor-free graph of girth at least~$g$ that is not in RIG$(\{H~:~H~\text{is}~K_t\text{-minor-free}\})$. 
\end{theorem}

The bridge between induced-minor-freeness and minor-freeness (if it exists) is not given straightforwardly by region intersection graphs. Our construction for proving \cref{thm:main} is an extension of the so-called \emph{Pohoata--Davies grids}~\cite{Davies22,Pohoata14} (see~\cref{fig:pohoata-grid}), a key family of graphs in the study of induced subgraphs and tree-decompositions.

Hopefully, our construction steers the search for a~link between induced-minor-freeness and minor-freeness in a~more fortunate direction.

\begin{figure}[h!]
	\centering
	\begin{tikzpicture}[vertex/.style={draw,circle,inner sep=0.03cm}]
		\def\k{6}
		\pgfmathtruncatemacro\km{\k - 1}
		\def\sv{0.5}
		\def\sh{0.8}
		
		\foreach \i in {1,...,\k}{
			\foreach \j in {1,...,\k}{
				\node[vertex] (a\i\j) at (\i * \sh, \j *\sv) {};
			}
		}
		\foreach \i in {1,...,\k}{
			\node[vertex] (s\i) at (\i * \sh, \k * \sv + \sv) {};
		}
		
		\foreach \i [count = \ip from 2] in {1,...,\km}{
			\foreach \j in {1,...,\k}{
				\draw (a\i\j) -- (a\ip\j) ;
			}
		}
		
		\foreach \i in {1,...,\k}{
			\foreach \j in {1,...,\k}{
				\draw (s\i) to [bend left = 20] (a\i\j) ;
			}
		}
	\end{tikzpicture}
	\caption{The Pohoata--Davies $6 \times 6$ grid.}
	\label{fig:pohoata-grid}
\end{figure}
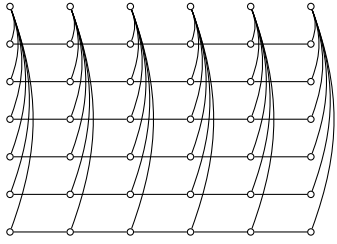

\section{Preliminaries}\label{sec:prelim}

Given an integer~$i$, we denote by $[i]$ the set of integers that are at least 1 and at~most~$i$. 

\subsection{Standard graph-theoretic notation}

We denote by $V(G)$ and $E(G)$ the set of vertices and edges of a graph $G$, respectively.
A~graph $H$ is a~\emph{subgraph} of a~graph $G$ if $H$ can be obtained from $G$ by vertex and edge deletions.
Graph~$H$ is an~\emph{induced subgraph} of $G$ if $H$ is obtained from $G$ by vertex deletions only.
For $S \subseteq V(G)$, the \emph{subgraph of $G$ induced by $S$}, denoted $G[S]$, is obtained by removing from $G$ all the vertices that are not in $S$.
Then $G-S$ is a short-hand for $G[V(G)\setminus S]$.

A~set $X \subseteq V(G)$ is connected (in $G$) if $G[X]$ has a~single connected component.
The \emph{girth} of a~graph is the number of vertices of one of its shortest cycles, and $\infty$ if the graph is acyclic.
A~graph class is \emph{weakly sparse} if it excludes $K_{t,t}$ as a~subgraph for some finite integer~$t$.
A~\emph{balanced separator} of an $n$-vertex graph $G$ is a~set $X \subseteq V(G)$ such that $G-X$ has no connected component on more than $n/2$ vertices. 

If $G$ is a~graph and $\ell$ is a~positive integer, then $G^{(\ell)}$ denotes the \emph{$\ell$-subdivision} of $G$ (replacing every edge of $G$ by a~path with $\ell+1$ edges), and $\ell G$ denotes the graph obtained from $\ell$ disjoint copies of $G$. We call the original vertices of $V(G)$ in $G^{(\ell)}$ \emph{branching vertices}, and the added vertices (which have degree 2) \emph{subdivision vertices}.

We say that two disjoint sets $X,Y\subseteq V(G)$ are \emph{anti-complete} if there is no edge in~$G$ with one end in~$X$ and the other in~$Y$.

The \emph{diameter of $G$} is defined as $\max_{u,v \in V(G)} d_G(u,v)$, where $d_G(u,v)$ is the number of edges in a~shortest path between $u$ and~$v$.
The \emph{weak diameter of $S$ in $G$} for $S \subseteq V(G)$ is equal to $\max_{u,v \in S} d_G(u,v)$.

\subsection{Tree-decomposition}

A \emph{tree-decomposition} of a graph $G$ is a collection $\TT=(W_x :x\in V(T))$ of subsets of $V(G)$ (called \emph{bags}) indexed by the vertices of a tree $T$, such that
\begin{compactitem}
\item for every edge $uv\in E(G)$, some bag $W_x$ contains both $u$ and $v$, and
\item for every vertex $v\in V(G)$, the set $\{x\in V(T):v\in W_x\}$ induces a non-empty (connected) subtree of $T$.
\end{compactitem}
  \medskip
  The \emph{width} of~$\TT$ is $\max\{|W_x| \colon x\in V(T)\}-1$. The \emph{treewidth} of $G$ is the minimum width of a tree-decomposition of $G$. The \emph{adhesion} of~$\TT$ is $\max\{|W_x\cap W_y| \colon xy\in E(T)\}$.
  The \emph{torso} of a bag $W_x$ (with respect to $\TT$), denoted by \emph{$\torso{G}{W_x}$}, is the graph obtained from the induced subgraph $G[W_x]$ by adding edges so that $W_x\cap W_y$ is a~clique for each edge $xy\in E(T)$. 
A~\emph{path-decomposition} is a tree-decomposition in which the underlying tree is a path, simply denoted by the corresponding sequence of bags $(W_1, \ldots, W_n)$.

\subsection{Minors, induced minors, and region intersection graphs}\label{subsec:m-im-rig}

A~graph $H$ is a \emph{minor} of a~graph $G$ if $H$ is isomorphic to a~graph that can be obtained from a subgraph of $G$ by contracting edges. Equivalently, $H$ is a minor of $G$ if there exists a model $\mathcal{M}=(X_v \subseteq G\colon v\in V(H))$ of $H$ in $G$ which is a collection of disjoint connected subgraphs of $G$ such that $X_u$ and $X_v$ are adjacent whenever $uv\in E(H)$. Each $X_u$ is called a \emph{branch set}. A~graph $H$ is an \emph{induced minor} of a~graph $G$ if $H$ is isomorphic to a~graph that can be obtained from an induced subgraph of $G$ by contracting edges. Equivalently, $H$ is an induced minor of $G$ if there is a model $\mathcal{M}=(X_v \subseteq G \colon v\in V(H))$ of $H$ in $G$ with the additional constraint that $X_u$ and $X_v$ are adjacent if and only if $uv\in E(H)$. A~graph~$G$ is \emph{$H$-minor-free} (resp.~\emph{$H$-induced-minor-free}) if~$H$ is not a~minor (resp.~an induced minor) of~$G$.

Recall that a graph $G$ is a~\emph{region intersection graph} over a~graph $H$ if there exists a collection $\mathcal{R}=(R_v \subseteq H \colon v\in V(G))$ of connected subgraphs of $H$ such that $uv \in E(G)$ if and only if $V(R_u)\cap V(R_v)\neq \emptyset$. 
We denote by RIG$(H)$ the class of graphs that are region intersection graphs over~$H$.
By extension, given a~graph class~$\mathcal C$, RIG$(\mathcal C)$ denotes the class of graphs that are region intersection graphs over some graph of~$\mathcal C$.

\subsection{Graph minor structure theorem}

The Graph Minor Structure Theorem of Robertson and Seymour~\cite{robertson2003graph} states that every $K_t$-minor-free graph has a tree-decomposition with bounded-size adhesion such that each torso can be constructed using three ingredients: graphs on surfaces, vortices, and apices.
To describe this formally, we need the following definitions. 

Let $G_0$ be a graph embedded in a surface $\Sigma$. 
A~closed disk $D$ in $\Sigma$ is \emph{$G_0$-clean} if its only points of intersection with $G_0$ are vertices of $G_0$ that lie on the boundary of $D$.
Let $x_1,\dots,x_b$ be the vertices of $G_0$ on the boundary of $D$ in the order around $D$. A \emph{$D$-vortex} (with respect to $G_0$) of a graph $H$ is a path-decomposition $(W_1, \dots, W_b)$ of $H$ such that $x_i\in W_i$ for each $i\in[b]$, and $V(G_0\cap H)=\{x_1,\dots,x_b\}$.

For integers $g,p,a \geq 0$ and $k\geq 1$, a graph $G$ is \emph{$(g,p,k,a)$-almost-embeddable} if for some set $Z\subseteq V(G)$ with $|Z|\leq a$, there are graphs $G_0,G_1,\dots,G_p$ such that:
\begin{compactitem}
	\item $G-Z = G_{0} \cup G_{1} \cup \cdots \cup G_p$,
	\item $G_{1}, \dots, G_p$ are pairwise vertex-disjoint,
	\item $G_{0}$ is embedded in a surface $\Sigma$ of Euler genus at most $g$,
	\item there are $p$ pairwise disjoint $G_0$-clean closed disks $D_1,\dots,D_p$ in $\Sigma$, and
	\item for $i\in[p]$, there is a $D_i$-vortex $(W_1,\dots,W_{b_i})$ of $G_i$ of width at most $k$.
\end{compactitem}
\medskip
The vertices in $Z$ are called \emph{apex} vertices---they can be adjacent to any vertex in $G$. A graph is \emph{$\ell$-almost-embeddable} if it is $(g,p,k,a)$-almost-embeddable for some $\ell \geq g,p,k,a$.
A graph is \emph{apex-free $\ell$-almost-embeddable} if it is $(g,p,k,0)$-almost-embeddable for some $\ell \geq g,p,k$.

\begin{theorem}[\cite{robertson2003graph}]\label{GMST}
	For every positive integer $t$, there exists an integer $\ell$ such that every $K_t$-minor-free graph has a tree-decomposition of adhesion at most $\ell$ such that each torso is $\ell$-almost-embeddable.
\end{theorem}

For every positive integer $n$, let $A_n$ denote the \emph{apex $n \times n$ grid}; that is, the graph obtained from the $n \times n$ grid by adding a universal vertex. The next theorem concerns the structure of apex-minor-free graphs.
The statement is implied by a characterization of apex-minor-free graphs \cite[Theorem~25, (6) $\Rightarrow$ (5)]{dujmovic2017layered}.

\begin{theorem}[\cite{dujmovic2017layered}]\label{NoApexMinor}
	For every positive integer $\ell$, there exists some integer $n$ such that every graph that has a tree-decomposition of adhesion at most $\ell$ where each torso is apex-free $\ell$-almost-embeddable is $A_n$-minor-free.
\end{theorem}

Finally, we will need the notion of \emph{clique-sum}.
Let $k$ be a~positive integer, $C_1=\{v_1,\dots,v_k\}$, a~clique in a~graph $G_1$, $C_2=\{w_1,\dots,w_k\}$, a clique in a~graph $G_2$.
A~\emph{\mbox{$k$-clique-sum}} of $G_1$ and $G_2$ is any graph $G$ obtained from the disjoint union of $G_1$ and $G_2$ by identifying $v_i$ and $w_i$ for each $i\in [k]$ and then possibly deleting some edges in $C_1$ $(=C_2)$.

\section{Proof of~\cref{thm:main}} 

In this section, we prove \cref{thm:main} first for graphs of girth~$5$. 
We then explain how the construction can be generalized so that the result holds for arbitrarily large girth.

\begin{theorem}\label{thm:main-girth5}
  For every positive integer $t$, there is a~$K_{\kk}^{(1)}$-induced-minor-free graph $G$ of girth~5 such that $G$ is not in RIG$(\{H~:~H~\text{is}~K_t\text{-minor-free}\})$. 
\end{theorem}

We fix any positive integer $t$.
By \cref{GMST}, there exists some integer $\ell := \ell(t)$ such that every $K_t$-minor-free graph has a~tree-decomposition of adhesion at most $\ell$ where each torso is $\ell$-almost-embeddable. By \cref{NoApexMinor}, there exists some integer $n := n(\ell)$ such that every graph that has a tree-decomposition of adhesion at most $\ell$ where each torso is apex-free $\ell$-almost-embeddable is $A_n$-minor-free. We may assume that $n\geq \ell+1$.
  We now construct our graph $G$. 

  \medskip
  
  \textbf{Construction of \bm{$G$}.}
  Since the $n \times n$ grid is $K_5$-minor-free, the apex $n\times n$ grid $A_n$ is $K_6$-minor-free.
  Let $B_n$ be $n A_n^{(1)}$, that is, the disjoint union of $n$ copies of the 1-subdivision of~$A_n$, also equal to the 1-subdivision of the disjoint union of $n$ copies of~$A_n$.
  We now set a~total order $\prec$ of $V(B_n)$, and a~traceable (i.e., admitting a~Hamiltonian path) spanning supergraph $B'_n$ of~$B_n$, whose Hamiltonian path defines the successor relation of $\prec$.

  The vertices of each copy of $A_n^{(1)}$ appear consecutively along~$\prec$.
  The graph $B'_n$ is obtained by adding to each copy of $A_n^{(1)}$ the red edges of~\cref{fig:Bn-Bpn}.
  Note that this includes an edge between the top-left vertex of the grid and the apex of the previous copy of~$A_n^{(1)}$ (leftmost vertex in the figure).
  The order $\prec$ within each~$A_n^{(1)}$ is given by the Hamiltonian path in blue, starting at the top-left vertex of the grid to the apex. 
  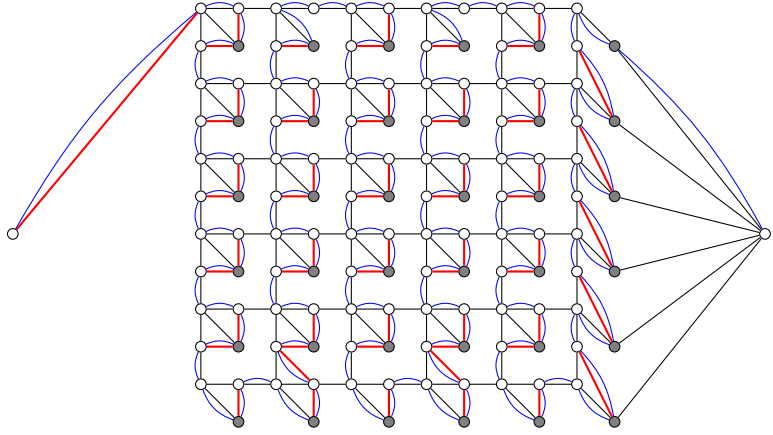
\begin{figure}[h!]
    \centering
    \begin{tikzpicture}[vertex/.style={draw,circle,inner sep=0.05cm}]
    \def\t{6}
    \pgfmathtruncatemacro\tm{\t-1}
    \pgfmathtruncatemacro\tmm{\tm - 1}
    \pgfmathtruncatemacro\ht{\t / 2}
    \def\s{1}
    \def\bb{25}

    \foreach \i in {1,...,\t}{
      \foreach \j in {1,...,\t}{
        \node[vertex] (v\i\j) at (\i * \s, \j * \s) {} ;
      }
    }

    \foreach \i in {1,...,\t}{ 
      \foreach \j in {1,...,\tm}{ 
        \pgfmathtruncatemacro\jnext{\j+1}
        \node[vertex] (h\i\j) at (\i * \s, \j * \s + 0.5*\s) {} ; 
        \draw (v\i\j) -- (h\i\j) ;
        \draw (h\i\j) -- (v\i\jnext) ;
      } 
    } 

    \foreach \i in {1,...,\tm}{ 
      \foreach \j in {1,...,\t}{ 
        \pgfmathtruncatemacro\inext{\i+1}
        \node[vertex] (w\i\j) at (\i * \s + 0.5*\s, \j * \s) {} ; 
        \draw (v\i\j) -- (w\i\j) ;
        \draw (w\i\j) -- (v\inext\j) ;
      } 
    }

    \foreach \i in {1,...,\t}{ 
      \foreach \j in {0,...,\tm}{ 
        \pgfmathtruncatemacro\inext{\i+1}
        \pgfmathtruncatemacro\jnext{\j+1}
        \node[vertex,fill=gray] (c\i\j) at (\i * \s + 0.5*\s, \j * \s + 0.5*\s) {} ;
        \draw (v\i\jnext) -- (c\i\j) ;
      } 
    }

    \foreach \i in {1,...,\tm}{ 
      \foreach \j in {1,...,\tm}{ 
        \draw[thick, red] (h\i\j) -- (c\i\j) ;
        \draw[blue] (h\i\j) to [bend left=\bb] (c\i\j) ;
      } 
    }

    \foreach \i in {1,...,\tm}{ 
      \foreach \j in {1,...,\tmm}{ 
        \pgfmathtruncatemacro\jnext{\j+1}
        \draw[thick, red] (w\i\jnext) -- (c\i\j) ;
        \draw[blue] (w\i\jnext) to [bend left=\bb] (c\i\j) ;
      } 
    }
    
    \foreach \i in {1,3,...,\tm}{ 
      \draw[thick, red] (w\i\t) -- (c\i\tm) ;
      \draw[blue] (w\i\t) to [bend left=\bb] (c\i\tm) ;
    }

     \foreach \i in {1,...,\tm}{ 
       \draw[thick, red] (w\i1) -- (c\i0) ;
       \draw[blue] (w\i1) to [bend left=\bb] (c\i0) ;
     }

     \foreach \i in {2,4,...,\tm}{ 
       \draw[thick, red] (w\i1) -- (h\i1) ;
       \draw[blue] (w\i1) to [bend left=\bb] (h\i1) ; 
     }

     \foreach \i [count=\im from 0] in {1,...,\tm}{
       \draw[thick, red] (c\t\im) -- (h\t\i) ;
       \draw[blue] (c\t\im) to [bend left=-16] (h\t\i) ;
     }

     \foreach \i in {1,...,\tm}{
       \foreach \j in {2,...,\t}{
         \draw[blue] (v\i\j) to [bend left=\bb] (w\i\j) ;
       }
     }
     \foreach \i in {2,...,\tm}{
       \foreach \j in {1,...,\tm}{
         \draw[blue] (v\j\i) to [bend left=\bb] (h\j\i) ;
       }
     }

     \foreach \i in {1,...,\t}{
       \draw[blue] (v\i1) to [bend left=-\bb] (c\i0) ;
     }
     
     \foreach \i in {1,3,...,\tm}{
       \draw[blue] (v\i1) to [bend left=\bb] (h\i1) ;
     }
     \foreach \i in {1,3,...,\tm}{
       \pgfmathtruncatemacro{\ip}{\i+1}
       \draw[blue] (w\i1) to [bend left=\bb] (v\ip1) ;
     }

      \foreach \i in {2,4,...,\tm}{
       \pgfmathtruncatemacro{\ip}{\i+1}
       \draw[blue] (w\i\t) to [bend left=\bb] (v\ip\t) ;
      }

     \foreach \i in {2,4,...,\tm}{
       \draw[blue] (v\i\t) to [bend left=\bb] (c\i\tm) ;
     }

     \foreach \i in {2,...,\t}{
       \pgfmathtruncatemacro{\im}{\i-1}
       \draw[blue] (v\t\i) to [bend left=-\bb] (c\t\im) ;
       \draw[blue] (v\t\i) to [bend left=-\bb] (h\t\im) ;
     }

     \node[vertex] (a1) at (-1.5 * \s, \t/2 * \s) {} ;
     \node[vertex] (a2) at (\t * \s + 2.5 * \s, \t/2 * \s) {} ;

     \draw[thick, red] (a1) -- (v1\t) ;
     \draw[blue] (a1) to [bend left=10] (v1\t) ;
     \foreach \i in {0,...,\tm}{
       \draw (c\t\i) -- (a2) ;
     }
     \draw[blue] (c\t\tm) to [bend left=10] (a2) ;
    \end{tikzpicture}
    \caption{The graphs $B_n, B'_n$ and the order $\prec$.
      We only represented one entire copy of~$A_n^{(1)}$.
      Black edges represent $B_n$.
      Together with the red edges, they form $B'_n$.
      Every vertex filled in gray is adjacent to the apex vertex to the right (we only drew some of these edges for legibility).
    The Hamiltonian path of $B'_n$ in blue defines the successor relation of~$\prec$.}
    \label{fig:Bn-Bpn}
  \end{figure}
  Like $B_n$, the graph $B'_n$ is also $K_6$-minor-free.
  The graph $B'_n$ is not part of the construction and we will only use it in the proof of~\cref{clm:imf}.
  
  To finish the construction, we add to $B_n$ the disjoint union of $n$ paths $P_1, \ldots, P_{n}$ of length $2|V(B_n)|-1$, and make for every $i \in [|V(B_n)|]$ and $j \in [n]$, the $(2i-1)$-st vertex of $P_j$, denoted by $p_{j, i}$, adjacent to the $i$-th vertex of $B_n$ along~$\prec$, denoted by $b_i$.
  Call $G$ the resulting graph. As a side note, if we replaced each copy of $A_n^{(1)}$ in $G$ by $K_1$, then the graph obtained is a Pohoata--Davies Grid (see \cref{fig:pohoata-grid}).

  The following three lemmas prove~\cref{thm:main-girth5}. 
  
  \begin{lemma}\label{clm:girth}
    $G$ has girth at~least~5.
  \end{lemma}
  \begin{proof}
    $B_n$ is the 1-subdivision of a~simple graph, hence has girth at~least~6.
    $G - V(B_n)$ is a~disjoint union of paths, thus does not contain any cycle.
    Any cycle going through $V(G) \setminus V(B_n)$ has at~least two consecutive edges within $G - V(B_n)$.
    We conclude as no distinct pair of vertices within the same connected component of~$V(G) \setminus V(B_n)$ shares a~neighbor in $V(B_n)$. 
  \end{proof}

  \begin{lemma}\label{clm:imf}
    $G$ is $K_{\kk}^{(1)}$-induced-minor-free.
  \end{lemma}
  \begin{proof}
    Assume for the sake of contradiction that $G$ admits $K_{\kk}^{(1)}$ as an induced minor.
    We will then build a~minor model of $K_{\kk}$ in $B'_n$, which, we know, does not exist.
    
    Let $\mathcal M$ be an induced minor model of~$K_{\kk}^{(1)}$ in $G$ such that
    \begin{compactitem}
    \item every branch set of a~subdivision vertex of~$K_{\kk}^{(1)}$ is a~singleton,
    \item if such a~singleton is on some~$P_j$ and its two neighbors on $P_j$ are in the two adjacent branch sets (one in each), then the singleton cannot be a~vertex $p_{j,i}$ (it has to be a~degree-2 vertex in between some $p_{j,i}$ and $p_{j,i+1}$), and
    \item each branch set is inclusion-wise minimal.
    \end{compactitem}
    It is easy to see that this can always be done.
    Let $X_1, \ldots, X_{\kk} \in \mathcal M$ be the branch sets corresponding to the branching vertices of~$K_{\kk}^{(1)}$.
    We denote by $\{s_{k,k'}\}$ the branch set (of the subdivision vertex) adjacent to $X_k$ and $X_{k'}$, for $k \neq k' \in [\kk]$.
    For each $k \in [\kk]$, let \[Y_k := (X_k \cap V(B'_n)) \cup \{b_i~:~\exists j \in [n],~p_{j,i} \in X_k~\text{and}~\nexists k' < k \in [\kk], j' \in [n],~p_{j',i} \in X_{k'}\},~\text{and}\]
    \[Y'_k := Y_k \cup \{s_{k,k'} \in V(B'_n) \setminus (Y_1 \cup \cdots \cup Y_{\kk})~:~k<k'\}.\]
    We now show that $Y'_1, \ldots, Y'_{\kk}$ is a~minor model of~$K_{\kk}$ in~$B'_n$.

    \begin{claim}\label{clm:pairwise-disjoint}
      The sets $Y'_1, \ldots, Y'_{\kk}$ are pairwise disjoint.
    \end{claim}
    \begin{proofofclaim}
      Suppose there exists some $b_i \in Y'_k \cap Y'_{k'}$ with $k<k'$.
      From the definition of $Y_1, \ldots, Y_{\kk}$ and $Y'_1, \ldots, Y'_{\kk}$, it should be that $b_i \in X_k$ and $p_{j,i} \in X_{k'}$ for some $j \in [n]$ or that $b_i \in X_{k'}$ and $p_{j,i} \in X_k$ for some $j \in [n]$.
      But that would make $X_k$ and $X_{k'}$ adjacent.
    \end{proofofclaim}

    To further show that the sets $Y'_1, \ldots, Y'_{\kk}$ are connected and pairwise adjacent in $B'_n$, we need the following notion and claims.
    An~\emph{interval} $I$ of some~$X_k$ is a~subset of consecutive positive integers such that there is a~connected component $J$ of $G[X_k \cap V(P_j)]$ for some $j \in [n]$ such that $\{i~:~p_{j,i} \in V(J)\}=I$.

    \begin{claim}\label{clm:no-nested-interval}
      For any $k \neq k' \in [\kk]$, any interval $I$ of $X_k$, and any interval $I'$ of $X_{k'}$, it cannot be that $I \subseteq I'$ (and symmetrically $I' \subseteq I$).
      Furthermore, at most one vertex of $\{b_i~:~i \in I \cap I'\}$ can be in a~branch set of~$\mathcal M$, namely $s_{k,k'}$.
    \end{claim}
    \begin{proofofclaim}
      If $I \subseteq I'$, then $X_k$ is a~subpath of~$P_j$ for some $j \in [n]$, as otherwise $X_k$ and $X_{k'}$ would be adjacent.
      But then $X_k$ has at~most two neighbors that are not neighbors of $X_{k'}$, a~contradiction to realize the $\kkmm$ branch sets adjacent to $X_k$ but not to $X_{k'}$.
      The rest of the claim follows because $\{s_{k,k'}\}$ is the only branch set adjacent to both $X_k$ and $X_{k'}$, and $X_k$ and $X_{k'}$ are non-adjacent.
    \end{proofofclaim}

    We can extend a~bit the previous claim. 

    \begin{claim}\label{clm:no-2-cover}
      For any pairwise distinct $k, k', k'' \in [\kk]$, any interval $I$ of~$X_k$, any interval $I'$ of~$X_{k'}$, and any interval $I''$ of~$X_{k''}$, it cannot be that $I \subseteq I' \cup I''$.
      Furthermore, if $s_{k',k''} = p_{j,i}$ for some $j \in [n]$, it cannot be that $I \subseteq I' \cup I'' \cup \{i\}$.
    \end{claim}
    \begin{proofofclaim}
      Again, any such inclusion would imply that $X_k$ is a~subpath of some~$P_j$.
      But then $X_k$ has at~most two neighbors that are not neighbors of $X_{k'} \cup X_{k''} (\cup \{s_{k',k''}\})$, a~contradiction to realize the $\kkmmm$ branch sets adjacent to $X_k$ but not to $X_{k'}$ nor $X_{k''}$.
    \end{proofofclaim}

    As $\mathcal M$ is minimal, \cref{clm:no-nested-interval,clm:no-2-cover} imply in particular that there is at~most one pair $I, I'$ of intervals of $X_k, X_{k'}$ with $I \cap I' \neq \emptyset$, per $k \neq k' \in [\kk]$.
    As another direct consequence of~\cref{clm:no-nested-interval,clm:no-2-cover}, we get the following.

    \begin{claim}\label{clm:no-triple}
      For any pairwise distinct $k, k', k'' \in [\kk]$, any interval $I$ of~$X_k$ and any interval $I'$ of~$X_{k'}$ such that $I \cap I' \neq \emptyset$ and $\min(I) < \min(I')$, there is no $i \in [\min(I')-1,\max(I)+1]$ such that $p_{j,i} \in X_{k''}$ for some $j \in [n]$ (or $b_i \in X_{k''}$). 
    \end{claim}

    The next two claims complete the proof.

    \begin{claim}\label{clm:connected}
      The sets $Y'_1, \ldots, Y'_{\kk}$ are connected in $B'_n$.
    \end{claim}
     \begin{proofofclaim}
       For any $k \in [\kk]$, and any $u, v \in Y'_k$, we exhibit a~$u$--$v$ path $P$ in $B'_n$ such that $V(P) \subseteq Y'_k$.
       (As we do not need to show that $P$ is a~path, we call it so, but only argue that it is a~walk, which is~sufficient.)
       Let $u' \in V(G)$ (resp.~$v' \in V(G)$) be the vertex in $(V(G) \setminus V(B'_n)) \cap X_k$ causing that $u \in Y'_k$ (resp~$v \in Y'_k$) if this applies, or $u' := u$ (resp.~$v' := v$), otherwise.
       Let $P'$ be a~$u'$--$v'$ path in $G$ such that $V(P) \setminus \{u',v'\} \subseteq X_k$.
       Observe that $u'$ and $v'$ may be equal to some $s_{k,k'}$ with $k < k'$, and thus not be in $X_k$ themselves.
       In which case, we simply run the following arguments with their neighbors in~$P'$ (which are in~$X_k$).
       Hence, we may as well suppose that $u', v' \in X_k$.

       If $P'$ is a~subpath of some $P_j$, we have $u' = p_{j,i}$ and $v' = p_{j,i'}$, no $X_{k'}$ with $k' < k$ contains some vertex $p_{j',i}$ or $p_{j',i'}$, and no other $X_{k'}$ contains $b_{i''}$ for any $i''$ between $i$ and~$i'$.
       By \cref{clm:no-nested-interval}, it means that for any integer $i''$ between $i$ and $i'$, no $X_{k'}$ with $k' < k$ contains some vertex $p_{j',i''}$, and no other $X_{k'}$ contains~$b_{i''}$.
       In particular, all such vertices $b_{i''}$ are in~$Y'_k$, and this makes the path $P$ between $u$ and~$v$.

       More generally, the path $P'$ alternates between maximal subpaths contained in $V(G) \setminus V(B_n)$ and maximal subpaths contained in $V(B_n)$.
       The latter are kept to build $P$.
       We then mimic each maximal subpath contained in $V(G) \setminus V(B_n)$ with a~path of $B'_n$ included in~$Y'_k$, with the appropriate endpoints. 
       By~\cref{clm:no-nested-interval}, in $P'$, every maximal subpath $p_{j,i} \ldots p_{j,i'}$ in $V(G) \setminus V(B_n)$ surrounded by two subpaths in $V(B_n)$ is such that the corresponding vertices $b_i \ldots b_{i'}$ are all in $Y'_k$, hence form the desired subpath of~$P$ in~$B'_n$.

       We finally move to the case when $P'$ starts with a~subpath $u'=p_{j,i} \ldots p_{j,i'} \neq v'$ maximal in $V(G) \setminus V(B_n)$; the case when $P'$ ends with such a~maximal subpath is dealt with symmetrically.
       We know that $b_{i'} \in X_k$, no $X_{k'}$ with $k' < k$ contains some vertex $p_{j',i}$, and no other $X_{k'}$ contains some vertex $b_{i''}$ where $i''$ is between $i$ and~$i'$.
       Thus by~\cref{clm:no-nested-interval}, all the vertices $b_i \ldots b_{i'}$ are in~$Y'_k$, the desired subpath of~$P$ in~$B'_n$.
     \end{proofofclaim}

     \begin{claim}\label{clm:pairwise-adjacent}
      The sets $Y'_1, \ldots, Y'_{\kk}$ are pairwise adjacent in $B'_n$.
     \end{claim}
     \begin{proofofclaim}
       For any $k \neq k' \in [\kk]$, let $u \in X_k, u' \in X_{k'}$ be such that $u s_{k,k'}, u' s_{k,k'} \in E(G)$.

       Assume first that $s_{k,k'} = b_i$ for some $i \in [|V(B'_n)|]$.
       If at~most one $\ell \in \{k, k'\}$ (thus, at~most one $\ell \in [\kk]$) is such that $p_{j,i} \in X_\ell$ for some $j \in [n]$, then either $s_{k,k'} \in Y'_k$ and $u' \in V(B'_n)$, or $s_{k,k'} \in Y'_{k'}$ and $u \in V(B'_n)$; so~$Y'_k$ and $Y'_{k'}$ are adjacent in $B'_n$.
       If, instead, there are $j, j'$ such that $p_{j,i} \in X_k$ and $p_{j',i} \in X_{k'}$, consider the intervals $I, I'$ of $X_k, X_{k'}$ associated to~$p_{j,i}, p_{j',i}$.
       \Cref{clm:no-triple} implies that there is some $i'$ such that $b_{i'} \in Y'_k$ and $b_{i'+1} \in Y'_{k'}$; so, again, $Y'_k$ and $Y'_{k'}$ are adjacent in $B'_n$.

       We next assume that $s_{k,k'} \in V(G) \setminus V(B'_n)$.

       First consider the case both $u$ and $u'$ are also in $V(G) \setminus V(B'_n)$.
       Let $I, I'$ be their associated interval, and assume without loss of generality that $\max(I) < \min(I')$.
       By the second item of the conditions satisfied by~$\mathcal M$, $\min(I')-\max(I)=1$.
       By~\cref{clm:no-2-cover}, there is no $k'' \in [\kk] \setminus \{k,k'\}$ such that $X_{k''}$ contains some vertex $p_{j,i}$ or $b_i$ with $i \in [\max(I),\min(I')]$.
       Besides, $X_k$ (resp.~$X_{k'}$) contains no vertex $p_{j,\min(I')}$ nor $b_{\min(I')}$ (resp.~$p_{j,\max(I)}$ nor $b_{\max(I)}$).
       Therefore, $b_{\max(I)} \in Y'_k$ and $b_{\min(I')}=b_{\max(I)+1} \in Y'_{k'}$, thus $Y'_k$ and $Y'_{k'}$ are adjacent in~$B'_n$.

       Finally consider, without loss of generality, that $s_{k,k'} = p_{j,i}$, $p_{j,i-1} \in X_k$, and $b_i \in X_{k'}$.
       By~\cref{clm:no-nested-interval}, there is no $\ell \in [\kk] \setminus \{k\}$ such that $X_\ell$ contains some vertex $p_{j',i-1}$ nor~$b_{i-1}$.
       Thus $b_{i-1} \in Y'_k$.
       As $b_i \in Y'_{k'}$, we have that $Y'_k$ and $Y'_{k'}$ are adjacent.
     \end{proofofclaim}

     \Cref{clm:pairwise-disjoint,clm:connected,clm:pairwise-adjacent} imply that $Y'_1, \ldots, Y'_{\kk}$ is a~$K_6$ minor model in~$B'_n$; a~contradiction.
  \end{proof}
 
  \begin{lemma}\label{lem:not-rig}
    For every $K_t$-minor-free graph $H$, $G$ is not a region intersection graph over $H$.
  \end{lemma}

  \begin{proof}
  Suppose, for contradiction, that there is a $K_t$-minor-free graph $H$ for which $G \in \text{RIG}(H)$. Let $\mathcal{R}=(R_v \subseteq H \colon v \in V(G))$ be a collection of connected subgraphs of $H$ such that $uv \in E(G)$ if and only if $V(R_u)\cap V(R_v)\neq \emptyset$. By \cref{GMST}, $H$ has a tree-decomposition $\TT = (W_x\colon x\in V(T))$ of adhesion at most $\ell$ where each torso is $\ell$-almost-embeddable. 
  
  We claim that there is an $x\in V(T)$ such that the bag $W_x$ intersects $V(R_v)$ for each vertex $v\in V(B_n)$. For each vertex $v\in V(B_n)$, the set $\{x\in V(T)\colon V(R_v)\cap W_x\neq \emptyset\}$ is a subtree of $T$. By the Helly property for subtrees, it suffices to show that any two such subtrees meet. 
  
  Assume, for contradiction, that there exist $u,v \in V(B_n)$ such that $V(R_u)$ and $V(R_v)$ do not intersect a common bag.
  Since $\TT$ has adhesion at most $\ell$, there is a set $S\subseteq V(H)$ with $|S|\leq \ell$ whose deletion separates $V(R_u)$ and $V(R_v)$.
  By construction, $G$ contains $n$ $u$--$v$ paths $u Q_1 v,\dots, u Q_n v$ with $Q_i\subseteq P_i$.
   So, for each $i \in [n]$, the connected subgraph ${Q}_i^{\star}=\bigcup(R_p \colon p\in V(Q_i))$ of $H$ connects $R_u$ to $R_v$, hence meets $S$.
  Since $Q_1,\dots, Q_{n}$ are pairwise anti-complete, the subgraphs ${Q}_1^{\star},\dots,{Q}_{n}^{\star}$ are pairwise vertex-disjoint, forcing $|S|\geq n \geq \ell+1$, a contradiction.
  
  Therefore, there is a bag $W_x$ in $\TT$ intersecting all regions $R_v$ for $v\in V(B_n)$.
  Since every adhesion set is a~clique in a~torso, $V(R_v)\cap W_x$ induces a connected subgraph $R_v'$ in $\torso{H}{W_x}$ for every $v\in V(B_n)$.
  However, there may be an edge $uv\in E(B_n)$ for which $V(R_u')\cap V(R_v')=\emptyset$.
  Nevertheless, since $V(R_u)\cap V(R_v)\neq \emptyset$, there is an adhesion set $S=W_x\cap W_y$ (for some edge $xy\in E(T)$) such that $R_u'\cap S\neq \emptyset$ and $R_v'\cap S\neq \emptyset$.
  Choose vertices $a\in V(R_u')\cap S$ and $b\in V(R_v')\cap S$.
  Then $ab\in E(\torso{H}{W_x})$.
  Add a vertex $w$ to $\torso{H}{W_x}$ adjacent to both $a$ and $b$ then include $w$ in the connected subgraphs $R_u'$ and $R_v'$.
  Repeating this procedure for every such edge produces a supergraph $H'$ of $\torso{H}{W_x}$ built by performing 2-clique-sums with triangles together with a collection $(R_v'\subseteq H' \colon v\in V(B_n))$ of~connected subgraphs in $H'$ that realizes $B_n$ as a~region intersection graph over $H'$.
 
  Let $Z \subseteq W_x$ be the set of apex vertices in~$\torso{H}{W_x}$.
  Since $|Z| \leq \ell$ and $B_n$ consists of~$n\geq \ell+1$ anti-complete copies of $A_n^{(1)}$, there exists a~copy of~$A_n^{(1)}$, denoted as $\widetilde{A}^{(1)}_n$, for which $\bigcup (V(R_x)\colon x\in V(\widetilde{A}^{(1)}_n))\cap Z = \emptyset$.
  Let $\widetilde{H}$ be the subgraph of $H'$ induced by $\bigcup (V(R_x)\colon x\in V(\widetilde{A}^{(1)}_n))$.
  Then $V(\widetilde{H}) \cap Z = \emptyset$.
  As such, $\widetilde{H}$ has a tree-decomposition with adhesion at~most~$2$ where one torso is an apex-free $\ell$-almost embeddable graph and the other torsos are triangles.
  By \cref{NoApexMinor}, $\widetilde{H}$~is \mbox{$A_n$-minor-free}.
  However, since $\widetilde{A}^{(1)}_n\in \text{RIG}(\widetilde{H})$ and $\widetilde{A}^{(1)}_n$ is isomorphic to $A^{(1)}_n$, \cref{No1Sub} implies $A_n$ is a~minor of~$\widetilde{H}$, giving us the desired contradiction. 
  \end{proof}

We now explain how to modify the construction in \cref{thm:main-girth5} to force the girth to be arbitrarily large.
Fix positive integer $g$.
Define $B_{g,n}$ to be $n A_n^{(g)}$, that is the disjoint union of $n$ copies of the $g$-subdivision of $A_n$.
Then $B_{g,n}$ has girth $3(g+1)$.
We define a total order $\prec$ of $V(B_{g,n})$ by using the same strategy of that given by~\cref{fig:Bn-Bpn}.
Similar to before, we add to $B_{g,n}$ the disjoint union of $n$ paths $P_1,\dots,P_n$ of length $g|V(B_{g,n})|-1$ and make, for every $i\in [|V(B_{g,n})|]$ and $j\in [n]$, the $(gi-1)$-st vertex of $P_j$ adjacent to the $i$-th vertex of $B_{g,n}$ along $\prec$.
Call the resulting graph $G_{g,n}$.
Since $G_{g,n}-B_{g,n}$ is a disjoint union of paths, it does not contain any cycle.
Any cycle going through $V(G_{g,n})\setminus V(B_{g,n})$ has at least $g-1$ consecutive edges within $G_{g,n}-V(B_{g,n})$.
Since no pair of vertices within the same connected component of $V(G_{g,n})\setminus V(B_{g,n})$ shares a neighbor in $V(B_{g,n})$, we conclude that every cycle in $G_{g,n}$ has length at least $g$.
Since \cref{clm:imf,lem:not-rig} also generalizes to $G_{g,n}$, this completes the proof of \cref{thm:main}.

\section*{Acknowledgments}

We thank Tuukka Korhonen, Daniel Lokshtanov, Rose McCarty, and Sebastian Wiederrecht for clarifying the origin of~\cref{conj:main}.

\bibliographystyle{abbrv}
\bibliography{main.bbl}

\begin{thebibliography}{10}

\bibitem{Alon90}
N.~Alon, P.~D. Seymour, and R.~Thomas.
\newblock A separator theorem for graphs with an excluded minor and its
  applications.
\newblock In H.~Ortiz, editor, {\em Proceedings of the 22nd Annual {ACM}
  Symposium on Theory of Computing, May 13-17, 1990, Baltimore, Maryland,
  {USA}}, pages 293--299. {ACM}, 1990.

\bibitem{Bonamy24}
M.~Bonamy, {\'{E}}.~Bonnet, H.~Déprés, L.~Esperet, C.~Geniet, C.~Hilaire,
  S.~Thomassé, and A.~Wesolek.
\newblock Sparse graphs with bounded induced cycle packing number have
  logarithmic treewidth.
\newblock {\em Journal of Combinatorial Theory, Series B}, 167:215--249, 2024.

\bibitem{BonnetHKM23}
{\'{E}}.~Bonnet, J.~Hodor, T.~Korhonen, and T.~Masa\v{r}{\'{\i}}k.
\newblock Treewidth is polynomial in maximum degree on weakly sparse graphs
  excluding a planar induced minor.
\newblock {\em CoRR}, abs/2312.07962, 2023.

\bibitem{CDDFGHHWWYl2024hadwiger}
R.~Campbell, J.~Davies, M.~Distel, B.~Frederickson, J.~P. Gollin, K.~Hendrey,
  R.~Hickingbotham, S.~Wiederrecht, D.~R. Wood, and L.~Yepremyan.
\newblock Treewidth, {H}adwiger number, and induced minors, 2024.

\bibitem{Dallard-3}
C.~Dallard, M.~Milanic, and K.~Storgel.
\newblock Treewidth versus clique number. {III.} {T}ree-independence number of
  graphs with a forbidden structure.
\newblock {\em J. Comb. Theory {B}}, 167:338--391, 2024.

\bibitem{Davies22}
J.~Davies.
\newblock Oberwolfach report 1/2022.
\newblock 2022.

\bibitem{Davies-pc}
J.~Davies.
\newblock Personal communication, 2024.

\bibitem{dujmovic2017layered}
V.~Dujmovi\'{c}, P.~Morin, and D.~R. Wood.
\newblock Layered separators in minor-closed graph classes with applications.
\newblock {\em J. Combin. Theory Ser. B}, 127:111--147, 2017.

\bibitem{GartlandThesis}
P.~Gartland.
\newblock {\em Quasi-Polynomial Time Techniques for Independent Set and Beyond
  in Hereditary Graph Classes}.
\newblock PhD thesis, UC Santa Barbara, 2023.

\bibitem{GP2023Metric}
A.~Georgakopoulos and P.~Papasoglu.
\newblock Graph minors and metric spaces.
\newblock {\em CoRR}, 2023.

\bibitem{KorhonenL24}
T.~Korhonen and D.~Lokshtanov.
\newblock Induced-minor-free graphs: Separator theorem, subexponential
  algorithms, and improved hardness of recognition.
\newblock In D.~P. Woodruff, editor, {\em Proceedings of the 2024 {ACM-SIAM}
  Symposium on Discrete Algorithms, {SODA} 2024, Alexandria, VA, USA, January
  7-10, 2024}, pages 5249--5275. {SIAM}, 2024.

\bibitem{Lee17}
J.~R. Lee.
\newblock Separators in region intersection graphs.
\newblock In C.~H. Papadimitriou, editor, {\em Proc. 8th {I}nnovations in
  {T}heoretical {C}omputer {S}cience {C}onference}, volume~67 of {\em LIPIcs},
  pages 1:1--1:8. Schloss Dagstuhl, 2017.

\bibitem{Lipton79}
R.~J. Lipton and R.~E. Tarjan.
\newblock A separator theorem for planar graphs.
\newblock {\em SIAM Journal on Applied Mathematics}, 36(2):177--189, 1979.

\bibitem{Lokshtanov-pc}
D.~Lokshtanov.
\newblock Personal communication, 2025.

\bibitem{McCarty-lectures}
R.~McCarty.
\newblock Structurally sparse graphs.
\newblock Lecture series at the Structural Graph Theory Bootcamp, Warsaw,
  \url{https://sites.google.com/view/strug/main}, 2023.

\bibitem{Pohoata14}
A.~C. Pohoata.
\newblock {\em Unavoidable induced subgraphs of large graphs}.
\newblock Senior theses, Department of Mathematics, Princeton University, 2014.

\bibitem{GraphMinors}
N.~Robertson and P.~D. Seymour.
\newblock {Graph Minors I--XXIII}.
\newblock {\em J. Combin. Theory Ser. B (and Journal of Algorithms)},
  1983--2012.

\bibitem{robertson2003graph}
N.~Robertson and P.~D. Seymour.
\newblock Graph minors. {XVI}. {E}xcluding a non-planar graph.
\newblock {\em J. Combin. Theory Ser. B}, 89(1):43--76, 2003.

\bibitem{Sinden66}
F.~W. Sinden.
\newblock {Topology of thin film RC circuits}.
\newblock {\em Bell System Technical Journal}, 45(9):1639--1662, 1966.

\bibitem{Wiederrecht23op}
S.~Wiederrecht.
\newblock Graph searching in {RIG}s.
\newblock In {\em Open Problems, GRASTA 2023 Workshop}, Bertinoro, Italy, 2023.
\newblock Section 1.3, Problem 3:
  \url{https://www-sop.inria.fr/teams/coati/events/grasta2023/slides/Grasta23_OpenProblems.pdf}.

\end{thebibliography}

\end{document}